\documentclass[a4paper]{article}
\usepackage{a4wide}
\usepackage{amsmath,amssymb,amsthm}
\usepackage{comment}
\usepackage{enumerate}
\usepackage{verbatim}
\usepackage{abstract}
\usepackage{mathpazo}
\usepackage{tikz}
\usetikzlibrary{decorations.pathreplacing}
\usepackage{hyperref}
\usepackage{thmtools}
\usepackage[T1]{fontenc}
\usepackage[utf8]{inputenc}

\usepackage[style=alphabetic,citestyle=alphabetic,sorting=nty,firstinits=true]{biblatex}
\bibliography{../references}

\usepackage{xcolor}
\definecolor{dark-red}{rgb}{0.4,0.15,0.15}
\definecolor{dark-blue}{rgb}{0.15,0.15,0.4}
\definecolor{medium-blue}{rgb}{0,0,0.5}
\hypersetup{
  colorlinks,
  linkcolor={dark-red},
  citecolor={dark-blue},
  urlcolor={medium-blue}
}

\DeclareFontFamily{OT1}{pzc}{}
\DeclareFontShape{OT1}{pzc}{m}{it}{<-> s * [1.15] pzcmi7t}{}
\DeclareMathAlphabet{\mathpzc}{OT1}{pzc}{m}{it}

\newcommand{\address}[1]{\vspace{-2em}\begin{center}{\footnotesize #1}\end{center}}

\usepackage{sectsty}
\sectionfont{\large\bfseries}
\subsectionfont{\normalsize\bfseries}

\newcommand{\mirror}[1]{\widetilde{#1}}

\newcommand{\A}{\mathpzc{A}}
\newcommand{\C}{\mathpzc{C}}
\newcommand{\HH}{\mathpzc{H}}
\newcommand{\PP}{\mathpzc{P}}
\newcommand{\Fac}{\mathpzc{F} \!}
\newcommand{\Rch}{\mathpzc{Rch}}
\newcommand{\Spe}{\mathpzc{Spe}}
\newcommand{\Bal}{\mathpzc{Bal} \!}
\newcommand{\Rev}{\mathpzc{Rev}}
\newcommand{\Pal}{\mathpzc{Pal} \!}
\newcommand{\Pri}{\mathpzc{Pri}}
\newcommand{\JSP}{\mathpzc{JSP}}
\newcommand{\Alphabet}{\mathpzc{Alph}}

\declaretheorem[numberwithin=section,refname={theorem,theorems},Refname={Theorem,Theorems}]{theorem}
\declaretheorem[sibling=theorem,style=definition]{definition}
\declaretheorem[sibling=theorem,name=Lemma]{lemma}
\declaretheorem[sibling=theorem,name=Proposition]{proposition}
\declaretheorem[sibling=theorem,name=Corollary]{corollary}
\declaretheorem[numbered=no,name=Conjecture]{conjecture}

\newcommand{\keywords}[1]{\par\noindent{\footnotesize{\em Keywords\/}: #1}}

\begin{document}
  \title{Introducing Privileged Words: Privileged Complexity of Sturmian Words}
  \author{Jarkko Peltomäki\\
          \small \href{mailto:jspelt@utu.fi}{jspelt@utu.fi}}
  \date{}
  \maketitle
  \address{Turku Centre for Computer Science TUCS, 20520 Turku, Finland\\
  University of Turku, Department of Mathematics and Statistics, 20014 Turku, Finland}

  \noindent
  \hrulefill
  \begin{abstract}
    \vspace{-1em}
    \noindent
    In this paper we study the class of so-called \emph{privileged words} which have been
    previously considered only a little. We develop the basic properties of privileged words,
    which turn out to share similar properties with palindromes. Privileged words are
    studied in relation to previously studied classes of words, \emph{rich words},
    \emph{Sturmian words} and \emph{episturmian words}. A new characterization of Sturmian words is given in terms of
    \emph{privileged complexity}. The privileged complexity of the Thue-Morse word is also briefly
    studied.
    \vspace{1em}
    \keywords{combinatorics on words, sturmian words, palindromes, privileged words, return words, rich words}
    \vspace{-1em}
  \end{abstract}
  \hrulefill

  \section{Introduction}
  This work concerns a new class of words named privileged words which have previously been researched only a little. The
  motivation for defining these words comes from the research of so-called rich words \cite{2009:palindromic_richness}
  which are words having maximum number of distinct palindromes (thus the name, rich words are rich in palindromes).
  An important property of rich words is that a word is rich if and only if every complete first return to a palindrome is
  a palindrome. It's equivalent to say “every palindrome is a complete first return to a shorter palindrome”. By a slight
  alteration of this condition we define privileged words: a word is privileged if it's a complete first return to a
  shorter privileged word. Moreover we need to define that the empty word and the letters of the alphabet are privileged.
  The effect of this modification is that every word is rich in privileged words, i.e. every word $w$ has exactly $|w|+1$
  distinct privileged factors whereas a rich word $w$ has exactly $|w|+1$ distinct palindromes (there exist words which
  have strictly less palindromic factors). It turns out that privileged words and palindromes have some similar properties.
  This paper introduces the basic properties of privileged words, and questions regarding so-called privileged complexity
  of Sturmian words, episturmian words and the Thue-Morse word are studied.

  After introducing the notations and definitions, in \autoref{sec:privileged_words} privileged words and their
  basic properties are presented. These basic results emphasize the analogue between palindromes and privileged words.
  Moreover privileged words are studied in relation to rich words.

  \autoref{sec:return_factors} studies the number of distinct privileged factors in finite words. There's also discussion
  how privileged words fit into a recent work of G. Fici and Z. Lipták \cite{2012:words_with_the_smallest_number_of_closed_factors}.

  Various complexity functions of infinite words have been previously considered. In \autoref{sec:sturmian_words} the notion
  of privileged complexity is defined. This section contains the main result of this paper: a characterization of Sturmian
  words using privileged complexity. As a by-product of the methods used in the proof of the main result, we obtain with
  little extra effort some previously known results, namely the fact that Sturmian words are rich, and partially a result
  of X. Droubay and G. Pirillo concerning the palindromic complexity of Sturmian words
  \cite{1999:palindromes_and_sturmian_words}. The section is concluded with a brief study of the privileged complexity of
  episturmian words.

  The last section studies briefly the privileged complexity of the Thue-Morse word. It's proven that the Thue-Morse word
  doesn't contain a privileged factor of odd length greater than three. However the even case is left open. Some numerical
  data and a conjecture are provided.

  \section{Notation and Terminology}
  In this text, we denote by $A$ a finite \emph{alphabet}, which is a finite non-empty set of symbols. The elements
  of $A$ are called \emph{letters}. A (finite) \emph{word} over $A$ is a sequence of letters. To the empty sequence
  corresponds the \emph{empty word}, denoted by $\varepsilon$. The set of all finite words over $A$ is denoted by $A^*$.
  The set of non-empty words over $A$ is the set $A^+ := A^* \setminus \{\varepsilon\}$. A natural operation of words
  is concatenation. Under this operation $A^*$ is a free monoid over $A$. The letters occurring in the word $w$ form
  the \emph{alphabet of $w$} denoted by $\Alphabet(w)$. From now on we assume that binary words are over the alphabet
  $\{0,1\}$. For binary words we define the \emph{exchange operation}: $\hat{0} = 1$ and $\hat{1}=0$. Given a finite word
  $w = a_1 a_2 \cdots a_n$ of $n$ letters, we say that the \emph{length} of $w$, denoted by $|w|$, is equal to $n$. By
  convention the length of the empty word is $0$.  We also denote by $|w|_a$ the number of occurrences of the letter $a$
  in $w$. The set of all words of length $n$ over the alphabet $A$ is denoted $A^n$.

  An \emph{infinite word} $w$ over $A$ is a function from the natural numbers to $A$. We consider such a function as a
  sequence indexed by the natural numbers with values in $A$. We write consicely $w = a_1 a_2 a_3 \cdots$ with $a_i \in A$.
  The set of infinite words is denoted by $A^\omega$. The infinite word $w$ is said to be \emph{ultimately periodic} if
  it can be written in the form $w = uv^\omega = uvvv \cdots$ for some words $u,v \in A^*$, $v \neq \varepsilon$. If
  $u = \varepsilon$, then $w$ is said to be \emph{periodic}. An infinite word which is not ultimately periodic is said to
  be \emph{aperiodic}.

  A finite word $u$ is a \emph{factor} of the finite or infinite word $w$ if it can be written that $w = zuv$ for some
  $z \in A^*$ and $v \in A^* \cup A^\omega$. If $z = \varepsilon$, the factor $u$ is called a \emph{prefix} of $w$. If
  $v = \varepsilon$, then we say that $u$ is a \emph{suffix} of $w$. If word $u$ is both a prefix and a suffix of $w$, then
  $u$ is a \emph{border} of $w$. The set of factors of $w$ is denoted by $\Fac(w)$. The set $\Fac_n(w)$ is defined to
  contain all factors of $w$ of length $n$. A set of words $X$ is \emph{factorial} if every factor of $w$ is a member of
  $X$ for all $w \in X$. If $w = a_1 a_2 \cdots a_n$, then we denote $w[i,j] = a_i \cdots a_j$ whenever the choices of
  positions $i$ and $j$ make sense. This notion is extended to infinite words in a natural way. An \emph{occurrence} of $u$
  in $w$ is such a position $i$, that $w[i,i+|u|-1] = u$. If such a position exists, we say that $u$ \emph{occurs} in $w$.
  If $w$ has exactly one occurrence of $u$, then we say that $u$ is \emph{unioccurrent} in $w$. We say that a position
  $i$ \emph{introduces} a factor $u$ if $w[i - |u| + 1, i] = u$, and $u$ is unioccurrent in $w[1,i]$. A
  \emph{complete first return} to the word $u$ is a word starting and ending with $u$, and containing exactly two
  occurrences of $u$. A word which is a complete first return to some word is called a \emph{complete return word}. A
  \emph{complete return factor} is a factor of some word which is a complete return word.

  The \emph{reversal} $\mirror{w}$ of $w = a_1 a_2 \cdots a_n$ is the word $\mirror{w} = a_n \cdots a_2 a_1$. If
  $\mirror{w} = w$, then we say that $w$ is a \emph{palindrome}. By convention the empty word is a palindrome. The set of
  palindromes of $w$ is denoted by $\Pal(w)$. Moreover we define $\Pal_n(w) = \Pal(w) \cap \Fac_n(w)$.
  
  Let $w = au$ where $a \in A$ and $u \in A^*$. We define the \emph{circular shift operation} $T$ as follows: $T(w) = ua$.
  By applying this shift operation repeatedly we obtain at most $|w|$ distinct words, called the \emph{conjugates} of $w$.

  Let $A$ and $B$ be two alphabets. A \emph{morphism} from $A^*$ to $B^*$ is a mapping $f: A^* \to B^*$ such that
  $f(uv) = f(u)f(v)$ for all words $u, v \in A^*$. Because of this morphic property, the morphism $f$ is fully determined
  by its images on the letters. The morphism $f$ is said to be \emph{non-erasing} if for every $a \in A$, $f(a) \in A^+$. 
  A non-erasing morphism naturally extends to infinite words: for an infinite word $w = a_1 a_2 a_3 \cdots$,
  $f(w) = f(a_1) f(a_2) f(a_3) \cdots$. The morphism $f$ is \emph{prolongable} if there exists a letter $a$ such that
  $f(a) = aw$ for some $w \in A^+$. An infinite word $w$ may be a \emph{fixed point} of a morphism, i.e. $f(w) = w$. For a
  prolongable morphism $f$ we have that $f^n(a)$ is a prefix of $f^{n+1}(a)$ for all $n \geq 0$. Thus we obtain a unique
  fixed point $f^\omega(a) := \lim_{n \to \infty} f^n(a)$.

  Given an infinite word $w$ over the alphabet $A$ we say that a factor $u$ of $w$ is \emph{right special} (resp.
  \emph{left special}) if $ua$ and $ub$ (resp. $au$ and $bu$) are both factors of $w$ for some distinct letters $a$ and
  $b$. A factor that is both right and left special is called \emph{bispecial}.

  A set of binary words $X$ is \emph{balanced} if for all $n \geq 0$ and every word $u$ and $v$ of $X$ of length $n$ it
  holds that $||u|_1 - |v|_1| \leq 1$. A binary (finite or infinite) word is said to be balanced if its set of factors is balanced.

  \section{Privileged Words}\label{sec:privileged_words}
  Privileged words are a less known class of words which were recently introduced in
  \cite{2011:a_characterization_of_subshifts_with_bounded_powers}. We define the set $\Pri_A$, the set of
  \emph{privileged words over $A$}, recursively as follows:
  \begin{itemize}
    \item[-] $\varepsilon \in \Pri_A$,
    \item[-] $a \in \Pri_A$ for every letter $a$ in the alphabet,
    \item[-] if $|w| \geq 2$, then $w \in \Pri_A$ if $w$ is a complete first return to a shorter privileged word.
  \end{itemize}
  When the alphabet is known from context, we omit the subscript $A$. Given a word $w$, we denote
  \begin{align*}
    \Pri(w) = \{u \in \Fac(w) : u \, \text{is privileged}\}.
  \end{align*}
  The set $\Pri_n(w)$ is defined to contain all privileged factors of $w$ of length $n$.

  The first few binary privileged words are
  \begin{align*}
    \varepsilon, 0, 1, 00, 11, 000, 111, 010, 101.
  \end{align*}
  Not every privileged word needs to be a palindrome, for example the words $00101100$ and $0120$ are privileged, but not
  palindromic. However privileged words and palindromes have some analogous properties, as we shall soon see.

  \begin{lemma}\label{lem:privileged_prefix_suffix}
    Let $w$ be a privileged word, and $u$ its any privileged prefix (respectively suffix). Then $u$ is a suffix
    (respectively prefix) of $w$.
  \end{lemma}
  \begin{proof}
    If $|w| \leq 1$ or $u = w$, then the claim is clear. Suppose that $|w| \geq 2$ and $|u| < |w|$. By definition $w$ is a
    complete first return to a shorter privileged word $v$. If $|v| < |u|$, then by induction $v$ is a suffix of $u$, and
    thus $v$ would have at least three occurrences in $w$ which is impossible. If $u = v$, then the claim is clear. Finally
    assume that $|v| > |u|$, then by induction $u$ is a suffix of $v$, and thus a suffix of $w$. The proof in the case that
    the roles of prefix and suffix are reversed is symmetric.
  \end{proof}

  The above Lemma is the first analogue to palindromes: a palindromic prefix of a palindrome occurs also as a suffix.

  \begin{lemma}
    Let $w$ be a privileged word, and $u$ its longest proper privileged prefix (suffix). Then $w$ is a complete first
    return to $u$. In other words the longest proper privileged prefix (suffix) of $w$ is its longest proper privileged border.
  \end{lemma}
  \begin{proof}
    If $|w| \leq 1$, then the claim is clear. Suppose that $|w| \geq 2$, and that $w$ is a complete first return to
    privileged word $v$. Now if $|u| > |v|$, then $v$ is a prefix of $u$, and thus by \autoref{lem:privileged_prefix_suffix}
    also a suffix of $u$. Hence $w$ has at least three occurrences of $v$, a contradiction. Therefore $|u| \leq |v|$, and
    by the maximality of $u$, $u = v$, which proves the claim. The proof in the case that the roles of prefix and suffix
    are reversed is symmetric.
  \end{proof}

  \begin{lemma}
    Let $w$ be a privileged word, and suppose that it has border $u$. Then $u$ is privileged.
  \end{lemma}
  \begin{proof}
    If $|w| \leq 1$, the claim is clear. Suppose that $|w| \geq 2$, and that $w$ is a complete first return to privileged
    word $v$. Since $v$ is the longest proper border of $w$, we may assume that $|u| < |v|$. Now $u$ is a prefix of $v$, and
    since $u$ is a suffix of $w$ and $v$ is a suffix of $w$, also $u$ is a suffix of $v$. Thus $u$ is a border of $v$, and
    by induction, a privileged word.
  \end{proof}

  Palindromes share this property too: every border of a palindrome is a palindrome.

  The study of so-called rich words was initiated in \cite{2009:palindromic_richness}. Rich words are words having maximum
  number of distinct palindromic factors. In the following definition we count $\varepsilon$ as a palindromic factor.

  \begin{definition}
    A word $w$ is \emph{rich} if it has exactly $|w| + 1$ distinct palindromic factors. An infinite word is rich if its
    every factor is rich.
  \end{definition}

  Next we state a useful characterization of rich words proven in \cite{2009:palindromic_richness}.

  \begin{theorem}\label{thm:richness_characterization}
    For any finite or infinite word $w$, the following properties are equivalent:
    \begin{enumerate}[(i)]
      \item w is rich,
      \item every factor of $w$ which is a complete first return to a palindrome is itself a palindrome. \qed
    \end{enumerate}
  \end{theorem}

  The fact that the condition in the next Proposition is necessary was proved in
  \cite{2011:a_characterization_of_subshifts_with_bounded_powers}.

  \newpage
  \begin{proposition}\label{prp:rich_privileged=palindrome}
    Let $w$ be a word. Then $w$ is rich if and only if $\Pri(w) = \Pal(w)$.
  \end{proposition}
  \begin{proof}
    ($\Longrightarrow$)
    Suppose that the word $w$ is rich. The claim is clear for factors $u$ of length $|u| \leq 1$. Assume first that $u$,
    $|u| > 1$, is privileged. By definition $u$ is a complete first return to a shorter privileged word $v$. By induction
    $v$ is a palindrome, and hence $u$ is a complete first return to a palindrome, and is by
    \autoref{thm:richness_characterization} itself a palindrome.

    Suppose then that $u$ is a palindrome. Let $v$ be the longest proper palindromic prefix of $u$. Now $u$ is a complete
    first return to $v$. Otherwise $u$ would have a proper prefix which is a complete first return to $u$, and by
    \autoref{thm:richness_characterization} this prefix would be a longer proper palindromic prefix of $u$ than $v$. By
    induction it follows that $v$ is privileged, and thus $u$ too is a privileged word.

    ($\Longleftarrow$)
    Suppose now that $\Pri(w) = \Pal(w)$. Now let $q$ be a complete first return to a palindrome $p$ in $w$. By assumption
    $p$ is privileged, and thus $q$ too is privileged. Again by assumption $q$ is a palindrome, and the claim follows from
    \autoref{thm:richness_characterization}.
  \end{proof}

  As noted in \cite{2011:a_characterization_of_subshifts_with_bounded_powers}, privileged words are a “maximal generalization” of
  palindromes in the sense that every word is rich in privileged words, as is seen in \autoref{cor:all_rich_privileged}.

  \section{Privileged Words and Complete Return Factors}\label{sec:return_factors}
  Privileged words are special kind of complete return words. In this section we will prove that every word $w$ has $|w|+1$ distinct
  privileged factors. We will also state a characterization of those words whose all complete return factors are privileged.
  This characterization has already been done in \cite{2012:words_with_the_smallest_number_of_closed_factors}, but it seems that
  the authors missed the concept of privileged words, so we will briefly show the connection between privileged words and their work.

  The authors of \cite{2012:words_with_the_smallest_number_of_closed_factors} called complete return factors
  \emph{closed factors}, but here we stick with more conventional vocabulary. In this section we count the empty word and the
  letters of the alphabet (as complete returns to the empty word) as complete return factors.

  \begin{lemma}\label{lem:at_least_one_complete_return_factor}
    Let $w \in A^+$. Then every position of $w$ introduces at least one complete return factor of $w$.
  \end{lemma}
  \begin{proof}
    Consider the position $i$ of the word $w$, and the longest complete return factor $v$ ending in $i$. Factor $v$ exists
    since letters are complete return factors. We prove that $v$ is unioccurrent in $w[1,i]$, which proves the claim. Now
    if $v$ had been introduced earlier, say at position $j$, then the factor $w[j - |v| + 1, i]$ would be a complete first
    return to $v$ contradicting the maximality of $v$.
  \end{proof}

  \begin{corollary}\emph{\cite{2012:words_with_the_smallest_number_of_closed_factors}}
    Every word $w$ has at least $|w|+1$ complete return factors. \qed
  \end{corollary}

  A word $w$ might have more than $|w|+1$ complete return factors (and most words do). Consider for instance the word $w = 1^k 01^k 0$.
  Every position in $w$ except the last introduces exactly one new complete return factor, but the last position introduces $k+1$
  new complete return factors, yielding a total of $3k+2 = |w| + \frac{1}{2}|w| - 1$ complete return factors in $w$.
  
  \begin{corollary}\label{cor:all_rich_privileged}
    Every word $w \in A^*$ has exactly $|w| + 1$ distinct privileged factors, i.e. every word is rich in privileged words.
  \end{corollary}
  \begin{proof}
    If we replace “longest complete return factor” with “longest privileged factor” in the proof of
    \autoref{lem:at_least_one_complete_return_factor}, we obtain that every position of $w$ introduces at least one new
    privileged factor. Now if some position $i$ would introduce two privileged factors, say $u$ and $v$, with $|u| < |v|$,
    then by \autoref{lem:privileged_prefix_suffix} $u$ would also be a prefix of $v$, i.e. it wouldn't be unioccurrent in
    $w[1,i]$. This is a contradiction, and thus every position of $w$ introduces exactly one new privileged factor.
  \end{proof}

  From the proof we obtain the following facts:

  \begin{corollary}
    Let $w$ be a word. If some position in $w$ introduces exactly one complete return factor, then this factor is privileged.
    The word $w$ has exactly $|w|+1$ complete return factors if and only if its every complete return factor is privileged. \qed
  \end{corollary}

  In the article \cite{2012:words_with_the_smallest_number_of_closed_factors} words having the minimum number of complete
  return factors were considered. The authors called such words \emph{C-poor words}. Since the minimum number of complete return
  factors in a word $w$ is $|w|+1$, we have that a word is C-poor if and only if its every complete return factor is privileged.
  We are ready to state a characterization of C-poor words.

  \begin{proposition}\emph{\cite{2012:words_with_the_smallest_number_of_closed_factors}}
    Let $w$ be a word. Then the following are equivalent:
    \begin{enumerate}[(i)]
      \item $w$ is C-poor,
      \item every complete return factor of $w$ is privileged,
      \item $w$ doesn't contain as a factor a complete first return to $xy$ for distinct letters $x$ and $y$. \qed
    \end{enumerate}
  \end{proposition}

  It's worth noting that by part \emph{(iii)} C-poor words avoid all squares except squares of letters.
  If the word $w$ is binary, then it can be said more:

  \begin{proposition}\emph{\cite{2012:words_with_the_smallest_number_of_closed_factors}}
  Let $w$ be a binary word. Then the following are equivalent:
  \begin{enumerate}[(i)]
    \item $w$ is C-poor,
    \item every complete return factor of $w$ is a palindrome,
    \item $w$ is a conjugate of a word in $0^*1^*$. \qed
  \end{enumerate}
  \end{proposition}

  Now part \emph{(ii)} of the above Proposition actually says that for a binary C-poor word $w$, $\Pri(w) \subseteq \Pal(w)$.
  We obtain that $\Pri(w) = \Pal(w)$, since $|w|+1 = |\Pri(w)| \leq |\Pal(w)| \leq |w|+1$. Thus by
  \autoref{prp:rich_privileged=palindrome} the word $w$ must be rich (as was observed in
  \cite{2012:words_with_the_smallest_number_of_closed_factors}). If the alphabet is larger than two letters, then not
  every C-poor word is rich: for instance the word $0120$ is C-poor, but not rich.

  \section{Privileged Complexity, Sturmian Words and Episturmian Words}\label{sec:sturmian_words}
  In the study of infinite words many different so-called complexity functions have been considered. It's clearly
  of interest to try to count the number of distinct privileged words of length $n$ occurring in a finite or infinite word
  $w$, that is, to figure out the privileged complexity of words.

  \begin{definition}
    Let $w$ be a finite or infinite word. The \emph{privileged complexity function} which counts the number of distinct
    privileged factors of length $n$ in $w$ is defined as
    \begin{align*}
      \A_n(w) = |\Pri_n(w)|
    \end{align*}
    for all $n \geq 0$.
  \end{definition}

  \subsection{Sturmian Words}
  In this section we prove some basic results about the privileged complexity function function, and a characterization
  of Sturmian words using privileged complexity. First we need to discuss some related complexity functions.

  \emph{The factor complexity function} $\C_n(w)$ of the word $w$ counts the number of distinct factors of $w$ of length $n$,
  i.e. $\C_n(w) = |\Fac_n(w)|$. We state the following well-known Theorem (for a proof see Theorem 1.3.13 of
  \cite{2002:algebraic_combinatorics_on_words}).

  \begin{theorem}\label{thm:aperiodic_classical}
    An infinite word $w$ is aperiodic if and only if $\C_n(w) \geq n+1$ for all $n \geq 0$, i.e $w$ is aperiodic if and
    only if it has at least one right special factor of each length. \qed
  \end{theorem}

  \emph{Sturmian words} are characterized by the fact that they are the simplest infinite aperiodic words in terms of
  the complexity function. They are defined as follows:

  \begin{definition}
    An infinite word $w$ is \emph{Sturmian} if $\C_n(w) = n+1$ for all $n \geq 0$.
  \end{definition}

  For more information about Sturmian words see the Chapter 2 of \cite{2002:algebraic_combinatorics_on_words}.
  The next two Propositions are well-known (for proofs see Propositions 2.1.2 and 2.1.3 of
  \cite{2002:algebraic_combinatorics_on_words}). 

  \begin{proposition}\label{prp:balanced_upper_bound}
    Let $X$ be a factorial set of words. If $X$ is balanced, then $|X \cap \{0,1\}^n| \leq n+1$ for all $n \geq 0$. \qed
  \end{proposition}

  \begin{proposition}\label{prp:unbalanced_pair}
    Let $X$ be a factorial set of words. If $X$ is unbalanced, then there exists a unique minimal unbalanced pair of the
    form $(0x0, 1x1)$ in $X$ where the word $x$ is a palindrome. \qed
  \end{proposition}

  Sturmian words are also characterized as follows (see Theorem 2.1.5 of \cite{2002:algebraic_combinatorics_on_words}):

  \begin{theorem}
    An infinite binary word is Sturmian if and only if it's aperiodic and balanced. \qed
  \end{theorem}

  Sturmian words have numerous other characterizations. The characterization of interest here is the characterization
  in terms of the palindromic complexity function, due to X. Droubay and G. Pirillo \cite{1999:palindromes_and_sturmian_words}.
  \emph{The palindromic complexity function} $\PP_n(w)$ is defined as $\PP_n(w) = |\Pal_n(w)|$, it counts the
  number of distinct palindromes of length $n$ in $w$.

  \begin{theorem}\label{thm:sturmian_pal_characterization}\emph{\cite{1999:palindromes_and_sturmian_words}}
    An infinite word $w$ is Sturmian if and only if it has palindromic complexity
    \begin{align*}
      \PP_n(w) = 
        \begin{cases}
        1, & \text{if $n$ is even} \\
        2, & \text{if $n$ is odd}
        \end{cases}
    \end{align*}
    for all $n \geq 0$. \qed
  \end{theorem}

  We say that the word $w$ has the property $\PP_{Pal}(w)$ if the word $w$ satisfies the palindromic complexity
  of the above Theorem. We shall prove that the condition of this Theorem is sufficient after
  we have established a (similar) proof of \autoref{thm:sturmian_pri_characterization}.

  It's natural to ask what is the privileged complexity of Sturmian words, and if the answer
  to this question characterizes Sturmian words. This indeed is the case, and it's the main result of this section.

  \begin{theorem}\label{thm:sturmian_pri_characterization}
    An infinite word $w$ is Sturmian if and only if it has privileged complexity
    \begin{align*}
      \A_n(w) = 
        \begin{cases}
        1, & \text{if $n$ is even} \\
        2, & \text{if $n$ is odd}
        \end{cases}
    \end{align*}
    for all $n \geq 0$.
  \end{theorem}

  The proof of this theorem is based on two Lemmas \ref{lem:characterization_2} and \ref{lem:characterization_1}.
  To simplify notations, we say that the word $w$ has the property $\JSP(w)$ if the word $w$ satisfies the privileged
  complexity of the above Theorem.

  We will first prove that an infinite word $w$ having the property $\JSP(w)$ must be Sturmian. For this purpose we
  introduce the concepts of $Q$-property and $Q$-factors of words.

  \begin{definition}
    A \emph{$Q$-property} of words is defined to satisfy the following conditions:
    \begin{itemize}
      \item[-] $Q(\varepsilon)$ and $Q(a)$ hold for all letters $a \in A$,
      \item[-] for every position $i$ in every word there exists a factor with property $Q$ ending at position $i$,
      \item[-] every position in every word introduces at most one factor with property $Q$.
    \end{itemize}
    Factors with property $Q$ are called \emph{$Q$-factors}.
  \end{definition}

  \begin{definition}
    The \emph{$Q$-complexity function} of a word $w$ is defined as
    \begin{align*}
      \HH^Q_n(w) = |\{u \in \Fac_n(w) : Q(u) \}|.
    \end{align*}
    For a finite word $w$ we define $\HH^Q(w) = |\{u \in \Fac(w) : Q(u)\}|$.
  \end{definition}

  The following Lemma follows easily from the definition.

  \begin{lemma}\label{lem:q_upper_bound}
    Every word has at most $|w| + 1$ distinct $Q$-factors, i.e. $\HH^Q(w) \leq |w| + 1$. \qed
  \end{lemma}

  It's well-known and easy to see that every position in every word introduces at most one new palindrome (see Proposition
  2 of \cite{2001:episturmian_words_and_some_constructions_of_de}). Hence “being a palindrome” is a $Q$-property. From the
  proof of \autoref{cor:all_rich_privileged} we see that every position in every word introduces exactly one new privileged
  factor, and thus “being a privileged word” is a $Q$-property. Third possible $Q$-property could be “being a power of a
  letter”.

  \begin{lemma}
    Let $w \in A^+$ be a finite word with $|\Alphabet(w)| \geq 2$. Then $\HH^Q_n(w) = 0$ for some $2 \leq n \leq |w|$.
  \end{lemma}
  \begin{proof}
    We may assume that $w$ has exactly $|w| + 1$ distinct $Q$-factors, since otherwise
    clearly $\HH^Q_n(w) = 0$ for some $2 \leq n \leq |w|$. Now if $\HH^Q_n(w) > 0$ for all $1 \leq n \leq |w|$, then
    since we have $|w|$ positions in $w$, every position needs to introduce a distinct $Q$-factor of different length.
    This however is impossible since by assumption $|\Alphabet(w)| \geq 2$ positions introduce a $Q$-factor of length one.
  \end{proof}

  \begin{definition}
    If $\HH^Q_n(w) = 0$ for some integer $n$, then we say that $n$ is a \emph{vanishing index of $w$}.
  \end{definition}

  When it's said that $n$ is a vanishing index in the last $y$-block of $y^m$ in the next Lemma, we mean that
  $n$ is a vanishing index of $y^m$, and that $(m-1)|y| < n \leq m|y|$.

  \begin{lemma}\label{lem:q_periodic}
    Let $w = y^\omega$ be a periodic infinite word. Then either $\HH^Q_n(w) = 0$ for infinitely many $n$ or there exists
    such $k$ that $\HH^Q_n(w) = 1$ for all $n \geq k$.
  \end{lemma}
  \begin{proof}
    Let $r = |y|$. If there are infinitely many vanishing indices, then the claim is clear. Assume that there are only
    finitely many vanishing indices. Let $m$ be such an integer that the last vanishing index $n$ is in the last $y$-block
    of $y^m$. If no such integer $n$ exists, we set $m=0$ and $n=-1$. Now concatenating $y$ to $y^m$ introduces at most $r$
    new $Q$-factors of different length since every position can introduce at most one new $Q$-factor. Now there might be
    some vanishing indices in the last $y$-block of $y^{m+1}$. However when the next $y$ is concatenated to $y^{m+1}$ it
    must be that $\HH^Q_i(y^{m+2}) > 0$ for all $n+1 \leq i \leq (m+1)r$ since if $\HH^Q_j(y^{m+2}) = 0$ for some
    $n+1 \leq j \leq (m+1)r$, then since adding more $y$'s to the end of $y^{m+2}$ doesn't introduce any new factors of
    length $j$, it would be that $\HH^Q_j(w) = 0$ contradicting the maximality of $n$.

    Now if there are $s$ vanishing indices in the last $y$-block of $y^{m+1}$, then there are at least $s$ vanishing indices
    in the last $y$-block of $y^{m+2}$. Otherwise concatenating $y$ to $y^{m+1}$ would have introduced more than $r$
    $Q$-factors of different length which is impossible. It could be that the number of vanishing indices in the last
    $y$-block of $y^{m+2}$ increases, but such a phenomenom can occur at most $r$ times. Hence there exists an integer $m'$
    such that the last $y$-block of $y^{m'+d}$ has $s'$ vanishing indices for all $d \geq 0$. We claim that
    $\HH^Q_i(w) = 1$ for all $i > m'r$, which proves the claim. Clearly by the maximality of $n$ $\HH^Q_i(w) \geq 1$ for
    all $i > m'r$. Now concatenating $y$ to $y^{m''}$ for $m'' \geq m'$ must introduce $r$ new $Q$-factors of different
    length, since otherwise the number of vanishing indices in the last $y$-block of $y^{m''+1}$ would increase. Now if
    $\HH^Q_j(w) \geq 2$ for some $m'r < j$, then concatenating $y$ to $y^{m''}$ for some $m'' \geq m'$ would introduce at
    least two new $Q$-factors of the same length or it would introduce at least one $Q$-factor of some length already
    introduced. This is impossible, since such a concatenation would introduce less than $r$ $Q$-factors of different length.
  \end{proof}

  \begin{corollary}
    Let $w$ be an ultimately periodic infinite word. Then either $\HH^Q_n(w) = 0$ for infinitely many $n$ or there exists
    such $k$ that $\HH^Q_n(w) = 1$ for all $n \geq k$.
  \end{corollary}
  \begin{proof}
    Let $w = xy^\omega$. Adding the prefix $x$ to $y^\omega$ introduces only finitely many new $Q$-factors. Hence the
    claim follows from \autoref{lem:q_periodic}.
  \end{proof}

  As a consequence of the above Corollary and the discussion after \autoref{lem:q_upper_bound} we have the following
  two Corollaries.

  \begin{corollary}\label{cor:jsp_aperiodic}
    An infinite word $w$ having the property $\JSP(w)$ is aperiodic. \qed
  \end{corollary}

  \begin{corollary}\label{cor:pal_aperiodic}
    An infinite word $w$ having the property $\PP_{Pal}(w)$ is aperiodic. \qed
  \end{corollary}

  However there exist aperiodic infinite words $w$ having $\Pri_n(w) = 0$ for infinitely many $n$. One
  example is the Thue-Morse word. See \autoref{prp:t_no_odd}.

  To simplify notations, we define for a word $w$ the following properties:
  \begin{itemize}
    \item[-] $\Rch_n(w) \Longleftrightarrow   \Pri_n(w) = \Pal_n(w)$,
    \item[-] $\Spe_n(w) \Longleftrightarrow $ there exists a unique right special factor of length $n-1$,
    \item[-] $\Bal_n(w) \Longleftrightarrow $ all factors of $w$ of length $n$ are balanced,
    \item[-] $\Rev_n(w) \Longleftrightarrow $ for each factor $u$ of length $n$, also $\mirror{u}$ is a factor of $w$.
  \end{itemize}

  \begin{lemma}\label{lem:bal_return_words}
    Let $w$ be an infinite binary word for which $\Bal_m(w)$ holds for all $0 \leq m \leq n$. Then any right
    special factor of $w$ with length strictly less than $n$ has at most two complete return factors in $w$.
  \end{lemma}
  \begin{proof}
    First we reason that there exists at most one right special factor of length $m < n$. Suppose on the contrary that
    there are two right special factors of length $m$, say $u$ and $v$. Let $z$ be the longest common suffix of $u$ and
    $v$, i.e. $u = u' az$ and $v = v'\hat{a}z$ for some letter $a$. Since $u$ and $v$ are right special,
    $aza, \hat{a}z\hat{a} \in \Fac(w)$, contradicting $\Bal_m(w)$ for some $m \leq n$.

    Let then $u$ be a right special factor of $w$ with length $|u| < n$. Suppose that $v_1$ and $v_2$ are two distinct
    complete first returns to $u$ in $w$. Let $x$ be the longest common prefix of $v_1$ and $v_2$. We claim that $x = u$.
    Assume on the contrary that $|x| > |u|$. Then $v_1 = xav'_1$ and $v_2 = x\hat{a}v'_2$ for some letter $a$, and hence
    $x$ is a right special factor. Now the suffix of $x$ of length $|u|$ is right special, so by the reasoning in the
    beginning of the proof, $x$ has $u$ as a suffix. This however contradicts the fact that $v_1$ and $v_2$ are complete
    returns to $x$. Now if there was a third complete return to $u$ in $w$, then it would have a common prefix of length
    $|u| + 1$ with either $v_1$ or $v_2$, which is not possible by the above.
  \end{proof}


  \begin{proposition}\label{prp:balanced_rich}
    Let $w$ be an infinite binary word. If $\Bal_m(w)$ holds for all $0 \leq m \leq n$, then $\Rch_m(w)$ holds for all
    $0 \leq m \leq n + 1$. Specifically an infinite balanced binary word is rich.
  \end{proposition}
  \begin{proof}
    The claim clearly holds for $m \leq 1$. Assume then that $\Bal_m(w)$ holds for all $0 \leq m \leq k \leq n$. We will
    show that then $\Rch_{k+1}(w)$ holds. \\

    \noindent
    \textbf{Case 1.} $\Pri_{k+1}(w) \subseteq \Pal_{k+1}(w)$ \\
    Let $x \in \Pri_{k+1}(w)$. Then $x$ is a complete first return to a shorter privileged word $u$. By the induction
    hypothesis $u$ is a palindrome. If $u$ overlaps with itself in $x$ or $x = u^2$, then $x$ must be a palindrome. Assume
    that this is not the case.

    Now if $|u| = 1$, then $x = u \hat{u}^l u$ is a palindrome. Suppose that $|u| = 2$, so $u = aa$ for some letter $a$.
    Then $x = aa\lambda aa$ for some $\lambda \neq \varepsilon$. Hence there exists in $x$ a complete first return to $a$
    of the form $a \hat{a}^l a$ for some $l \geq 1$. By \autoref{lem:bal_return_words} the words $a\hat{a}^l a$ and $aa$
    are the only complete first returns to $a$. Then as $aa$ is not a factor of $a\lambda a$, it must be that
    $a\lambda a = (a \hat{a}^l)^t a$ for some $t \geq 1$. Thus $x$ is a palindrome.

    We may now assume that $|u| \geq 3$, Write $u = au'a$ for some letter $a$ and $u' \neq \varepsilon$. Note that $u'$ is
    a palindrome, and hence by the hypothesis privileged. Now $x = au'a\lambda au'a$ with $\lambda \neq \varepsilon$.
    Consider $z = u'a \lambda au'$, the center of $x$. We will prove that $z$ is a palindrome. From this it follows that
    $x$ too is palindromic. If $z$ is a complete first return to $u'$, then $z$ is privileged, and by hypothesis a
    palindrome. Assume then that $z$ contains at least three occurrences of $u'$. Now word $z$ has as a proper prefix
    a complete return to $u'$ which begins with $u'a$. Denote this prefix as $p$. As $x$ is a complete first
    return to $u = au'a$, the word $z$ doesn't have $au'a$ as a factor. Hence it now must have $au' \hat{a}$ as a factor.
    Therefore $z$ contains as a factor a complete first return to $u'$ beginning with $u'\hat{a}$. Denote
    this factor by $q$. For a better grasp of the situation see \autoref{fig:balanced_rich}. As $u'$ is right special,
    by \autoref{lem:bal_return_words} words $p$ and $q$ are the only
    complete return factors of $u'$ in $w$. Since $au'a$ and $\hat{a} u' \hat{a}$ are not factors of $z$ ($z$ is balanced), the
    occurrences of $p$ and $q$ must alternate in $z$. Since both $p$ and $q$ are palindromes as complete
    first returns to $u'$ and $z$ begins and ends with $p$, it follows that $z$ is a palindrome.\\

    \noindent
    \textbf{Case 2.} $\Pal_{k+1}(w) \subseteq \Pri_{k+1}(w)$ \\
    Let then $x \in \Pal_{k+1}(w)$ and $u$ its longest proper border. Now $u$ must be a palindrome, and hence by the
    induction hypothesis, a privileged word. The word $x$ must be a complete first return to $u$, since otherwise there
    would be a privileged proper prefix $v$ longer than $u$. That would be a contradiction with the maximality of $u$, since
    $v$ would also be a palindrome by the induction hypothesis. Therefore $x$ is privileged.\\

    \noindent
    The last claim follows now from \autoref{prp:rich_privileged=palindrome}.
  \end{proof}

  \begin{figure}[t]
  \caption{A picture clarifying the proof of \autoref{prp:balanced_rich}. Note that not all occurrences of $u'$
  need to be non-overlapping.}
  \label{fig:balanced_rich}
  \begin{center}
  \begin{tikzpicture}[
    word/.style={thick,draw=black!100},
    box/.style={thin,draw=black!100,minimum width=1cm,font=\small},
    k/.style={font=\small}]

    \node at (-0.5,-0.2) [font=\small] {$z$};

    \node (R1) at (0.35,0) [box] {$u'$};
    \node (A1) [right of=R1,xshift=-0.3cm,k] {$a$};
    \node (D1) [right of=A1,xshift=-0.6cm,k] {$\ldots$};
    \node (A2) [right of=D1,xshift=-0.6cm,k] {$a$};
    \node (R2) [right of=A2,xshift=-0.3cm] [box] {$u'$};
    \node (A3) [right of=R2,xshift=-0.3cm,k] {$\hat{a}$};
    \node (D2) [right of=A3,xshift=-0.6cm,k] {$\ldots$};
    \node (A3) [right of=D2,xshift=-0.6cm,k] {$\hat{a}$};
    \node (R3) [right of=A3,xshift=-0.3cm] [box] {$u'$};

    \node (D3) [right of=R3,xshift=0.5cm,k] {$\ldots$};

    \node (R4) [right of=D3,xshift=0.5cm] [box] {$u'$};
    \node (A4) [right of=R4,xshift=-0.3cm,k] {$\hat{a}$};
    \node (D4) [right of=A4,xshift=-0.6cm,k] {$\ldots$};
    \node (A5) [right of=D4,xshift=-0.6cm,k] {$\hat{a}$};
    \node (R5) [right of=A5,xshift=-0.3cm] [box] {$u'$};
    \node (A6) [right of=R5,xshift=-0.3cm,k] {$a$};
    \node (D5) [right of=A6,xshift=-0.6cm,k] {$\ldots$};
    \node (A7) [right of=D5,xshift=-0.6cm,k] {$a$};
    \node (R6) [right of=A7,xshift=-0.3cm] [box] {$u'$};

    \draw [decorate,decoration={brace,amplitude=6pt},xshift=0,yshift=0] (3.05,-0.4) -- (-0.16,-0.4) node [black,midway,yshift=-0.5cm] {\footnotesize $p$};
    \draw [decorate,decoration={brace,amplitude=6pt},xshift=0,yshift=0] (2.05,0.42) -- (5.25,0.42) node [black,midway,yshift=0.5cm] {\footnotesize $q$};

    \draw [decorate,decoration={brace,amplitude=6pt},xshift=0,yshift=0] (7.25,0.42) -- (10.45,0.42) node [black,midway,yshift=0.5cm] {\footnotesize $q$};
    \draw [decorate,decoration={brace,amplitude=6pt},xshift=0,yshift=0] (12.68,-0.4) -- (9.47,-0.4) node [black,midway,yshift=-0.5cm] {\footnotesize $p$};

    \coordinate (A) at (-0.16,-0.25);
    \coordinate (B) at (12.66,-0.25);
    \draw[word,-] (A) -- (B);
  \end{tikzpicture}
  \end{center}
  \end{figure}
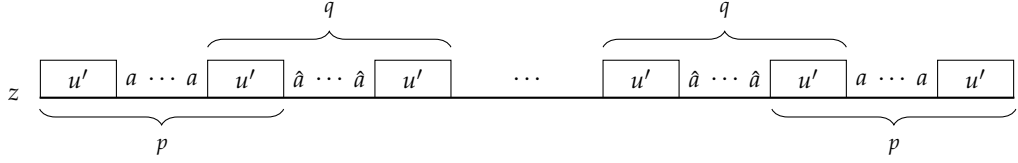

  We are now ready to prove the other direction of \autoref{thm:sturmian_pri_characterization}. The proof is
  similar to the proof of \autoref{thm:sturmian_pal_characterization} in \cite{1999:palindromes_and_sturmian_words}.

  \begin{lemma}\label{lem:characterization_2}
    An infinite word $w$ having the property $\JSP(w)$ is Sturmian.
  \end{lemma}
  \begin{proof}
    First of all since $\A_1(w) = 2$, $\Pri_1(w) = \{0,1\}$ and the word $w$ is thus binary. Hence $\varepsilon$ is the
    unique right special factor of length $0$. Clearly $\Rch_1(w), \Spe_1(w), \Bal_1(w)$ and $\Rev_1(w)$ hold. We will
    next assume that $\Rch_n(w), \Spe_n(w), \Bal_n(w)$ and $\Rev_n(w)$ hold for $n \geq 1$, and prove that $w$
    satisfies all these properties for $n + 1$. This proves the claim.\\

    \noindent
    \textbf{Case 1.} $\Rch_{n+1}(w)$ \\
    By the induction hypothesis $\Bal_m(w)$ holds for all $m \leq n$. From \autoref{prp:balanced_rich} it follows that also
    $\Rch_{n+1}(w)$ holds.\\

    \noindent
    \textbf{Case 2.} $\Bal_{n+1}(w)$ \\
    Assume on the contrary that $\Bal_{n+1}(w)$ doesn't hold. Then by \autoref{prp:unbalanced_pair} there exists a
    palindromic factor $x$ such that $0x0, 1x1 \in \Fac_{n+1}(w)$. By Case 1 we know that $\Pri_{n+1}(w) = \Pal_{n+1}(w)$,
    and since palindromes $0x0, 1x1 \in \Pal_{n+1}(w)$, by assumption, $n+1$ is odd, and hence also $n-1$ is odd. Again by
    assumption $\Pal_{n-1}(w) = \{x,t\}$ where $x \neq t$. Moreover $x$ is a bispecial factor of length $n-1$. Since
    $\Spe_n(w)$ holds, we conclude that $t$ isn't right special. Now $ta \in \Fac_n(w)$ for some letter $a$. By the
    property $\Rev_n(w)$, also $at \in \Fac_n(w)$. Since $t$ isn't right special, $ata \in \Fac_{n+1}(w)$. Thus
    $\{0x0, 1x1, ata\} \subseteq \Pal_{n+1}(w) = \Pri_{n+1}(w)$, contradicting the fact that $\A_{n+1}(w) = 2$.\\

    \noindent
    \textbf{Case 3.} $\Spe_{n+1}(w)$ \\
    Since by \autoref{cor:jsp_aperiodic} the word $w$ is aperiodic, we have that it has at least one right special factor
    of length $n$ (\autoref{thm:aperiodic_classical}). Arguing as in the first paragraph of the proof of
    \autoref{lem:bal_return_words} we see that it has at most one right special factor of length $n$.\\

    \noindent
    \textbf{Case 4.} $\Rev_{n+1}(w)$ \\
    Denote $\mirror{\Fac}_n(w) = \{\mirror{w} : w \in \Fac_n(w)\}$.
    Consider the set $X = \bigcup_{i=0}^{n+1} \Fac_i(w) \cup \mirror{\Fac}_i(w)$. This set is balanced since otherwise by
    \autoref{prp:unbalanced_pair} there would exists palindromes $0x0, 1x1 \in X$, and hence (since these
    words are palindromes) $0x0, 1x1 \in \Fac_m(w)$ for some $m \leq n+1$. This is a contradiction with the induction
    hypothesis or the Case 2. Thus by \autoref{prp:balanced_upper_bound} $|X \cap \{0,1\}^{n+1}| \leq n+2$. On the other
    hand by \autoref{thm:aperiodic_classical} $|X \cap \{0,1\}^{n+1}| \geq C_{n+1}(w) = n+2$. Thus
    $|X \cap \{0,1\}^{n+1}| = n+2$, and it must therefore be that $\Fac_{n+1}(w) = \mirror{\Fac}_{n+1}(w)$ which means that
    $\Rev_{n+1}(w)$ is satisfied.
  \end{proof}

  For the converse of \autoref{lem:characterization_2} we state the immediate Corollary of \autoref{prp:balanced_rich}.
  For another proof see Corollary 4 of \cite{2001:episturmian_words_and_some_constructions_of_de}.

  \begin{corollary}\label{cor:sturmian_rich}
    Sturmian words are rich. \qed
  \end{corollary}

  From this we easily deduce the converse result:
 
  \begin{lemma}\label{lem:characterization_1}
    Sturmian word $w$ has the property $\JSP(w)$.
  \end{lemma}
  \begin{proof}
    By \autoref{cor:sturmian_rich} the Sturmian word $w$ is rich. Next, \autoref{prp:rich_privileged=palindrome} says that
    $\Pri(w) = \Pal(w)$, and hence by \autoref{thm:sturmian_pal_characterization} the word $w$ has the property $\JSP(w)$.
  \end{proof}

  Lemmas \ref{lem:characterization_2} and \ref{lem:characterization_1} establish
  \autoref{thm:sturmian_pri_characterization}. The proof of \autoref{lem:characterization_2} with minor
  modifications proves too that an infinite word $w$ having the property $\PP_{Pal}(w)$ is Sturmian: the Case
  1 is omitted, and the Case 2 needs to be slightly adjusted. Otherwise the proof can be kept
  intact, \autoref{cor:pal_aperiodic} ensures that an infinite word having property $\PP_{Pal}(w)$ must be aperiodic.

  Note that not every $Q$-complexity function characterizes Sturmian words. Take $Q$ to be the property
  “being a power of a letter”. Now for a Sturmian word $w$ either $00 \notin \Fac(w)$ or $11 \notin \Fac(w)$.
  By symmetry assume that $11 \notin \Fac(w)$. It's also well-known that there exists a $k$ such that
  $0^k \in \Fac(w)$, but $0^{k+1} \notin \Fac(w)$. Hence the $Q$-complexity of $w$ would be
  \begin{align*}
      \HH^Q_n(w) = 
      \begin{cases}
      1, & \text{if $n=0$ or $2 \leq n \leq k$} \\
      2, & \text{if $n=1$}\\
      0, & \text{otherwise}.
      \end{cases}
  \end{align*}
  However for instance the non-Sturmian word $(0^k1)^\omega$ has the same $Q$-complexity function.

  \subsection{Episturmian Words}
  We conclude \autoref{sec:sturmian_words} by considering briefly the privileged complexity of so-called episturmian words,
  which are a generalization of Sturmian words to arbitrary alphabet. Episturmian words have a rich theory,
  for more about these intriguing words, see the foundational paper \cite{2001:episturmian_words_and_some_constructions_of_de}
  and the survey \cite{2009:episturmian_words_a_survey}.

  \begin{definition}
    An infinite word $w$ over alphabet $A$ ($|A| \geq 2$) is \emph{episturmian} if $w$ has the property $\Rev_n(w)$ and has
    at most one right special factor of length $n$ for all $n \geq 0$.
  \end{definition}

  \begin{definition}
    Episturmian word $w$ is \emph{$A$-strict} if for each $n \geq 0$ there exists a unique right special factor $u$ of
    length $n$ and $ua \in \Fac(w)$ for all $a \in A$.
  \end{definition}

  Actually $A$-strict episturmian words over the alphabet $A$ are exactly the so-called \emph{Arnoux-Rauzy}
  words over $A$. Note that $A$-strict episturmian words are aperiodic, and if $A = \{0,1\}$, then these
  words are exactly the Sturmian words.

  In the Corollary 2 of \cite{2001:episturmian_words_and_some_constructions_of_de} it was proved that
  episturmian words are rich. Therefore we may proceed as we did with Sturmian words: the palindromic
  complexity of episturmian words gives us their privileged complexity by \autoref{prp:rich_privileged=palindrome}.

  The following is Theorem 4.4 in \cite{2002:episturmian_words_and_episturmian_morphisms}.

  \begin{theorem}{\emph{\cite{2002:episturmian_words_and_episturmian_morphisms}}}\label{thm:episturmian_pal_compl}
    Let $w$ be a $A$-strict episturmian word over the alphabet $A$. Then $w$ has palindromic complexity
    \begin{align*}
      \PP_n(w) = 
        \begin{cases}
        1, & \text{if $n$ is even} \\
        |A|, & \text{if $n$ is odd}
        \end{cases}
    \end{align*}
    for all $n \geq 0$. \qed
  \end{theorem}

  Thus we have the following result:

  \begin{theorem}\label{thm:episturmian_pri_compl}
    Let $w$ be a $A$-strict episturmian word over the alphabet $A$. Then $w$ has privileged complexity
    \begin{align*}
      \A_n(w) = 
        \begin{cases}
        1, & \text{if $n$ is even} \\
        |A|, & \text{if $n$ is odd}
        \end{cases}
    \end{align*}
    for all $n \geq 0$. \qed
  \end{theorem}

  However the privileged complexity of $A$-strict episturmian words doesn't characterize them when $|A| > 2$.
  Words coding $r$-interval exchange transformations are a class of rich words which satisfy the palindromic
  complexity of \autoref{thm:episturmian_pal_compl} \cite{2007:factor_versus_palindromic_complexity_of_uniformly}.
  Being rich they also have the same
  privileged complexity as episturmian words. However words coding $r$-interval exchange transformations are
  not episturmian when $r>2$ (here $r$ is the number of letters). One example of such a word is the fixed point of the following morphism
  \begin{align*}
    \alpha:
    \begin{array}{l}
      a \mapsto c, \\
      b \mapsto ca, \\
      c \mapsto caba.
    \end{array}
  \end{align*}
  The fixed point isn't episturmian since both letters $a$ and $c$ are right special, but it satisfies the complexities
  of Theorems \ref{thm:episturmian_pal_compl} and \ref{thm:episturmian_pri_compl}.

  Actually not even both factor complexity and privileged complexity of $A$-strict episturmian words characterizes them since
  words coding $r$-interval exchange transformations have the same factor complexity as episturmian words
  \cite{2007:factor_versus_palindromic_complexity_of_uniformly}.

  \section{Privileged Complexity and the Thue-Morse Word}\label{sec:thue_morse}

  In this section we investigate briefly the privileged complexity of the Thue-Morse word. The infinite Thue-Morse
  word $t$ is defined as the fixed point of the morphism $\varphi$:
  \begin{align*}
    \varphi:
    \begin{array}{l}
      0 \mapsto 01, \\
      1 \mapsto 10.
    \end{array}
  \end{align*}
  For more information about the Thue-Morse word, see Chapter 2 of \cite{1983:combinatorics_on_words}.
  The word $t$ has the following well-known property:

  \begin{theorem}
    The Thue-Morse word $t$ is overlap free. \qed
  \end{theorem}

  \begin{lemma}\label{lem:t_00_11}
    Let $w$ be a non-empty even length privileged factor of $t$. Then $00$ or $11$ is a factor of $w$.
  \end{lemma}
  \begin{proof}
    Using the fact that $t$ is overlap free, it can be easily shown that no factor of $t$ of length greater than four
    avoids factors $00$ and $11$. The only possible privileged factors of length two are $00$ and $11$. For privileged
    factors of length four, the possibilities are $0000$, $1111$, $0110$ and $1001$. Words $0000$ and $1111$ are not
    factors of $t$, but anyway the claim is proved.
  \end{proof}

  \begin{proposition}\label{prp:t_no_odd}
    The infinite Thue-Morse word $t$ doesn't have any privileged factors of length $n$, when $n$ is odd and $n \geq 5$.
  \end{proposition}
  \begin{proof}
    Let $w$ be an privileged factor of $t$ of odd length, which is a complete first return to a privileged word
    $u$. Denote $x = 01$ and $y = 10$.

    Assume first that $|u|$ is odd. Since $000$ and $111$ are not factors of $t$, it must be that $|u| > 1$. Moreover
    $|w| > 5$, since if $|w| = 5$, then the occurrences of $u$ in $w$ would need to overlap, and $t$ is
    overlap free. We need to only prove that $u$ can't be $010$ or $101$ (the privileged factors of $t$ of length three).
    We prove that $u$ can't be $010$, the other case is symmetric. 
    Assume first that the factorization of $w$ over $\{x,y\}$ matches from the beginning of $w$. So
    we are looking for a factor of $t$ starting and ending with $xx$, and containing no internal
    occurrences of $xx = 0101$ or $yy = 1010$. Using the fact that $t$ is overlap free, one can by inspection
    deduce that the only possibility is $xxyxx$. However $xxyxx$ is not a factor of $t$. Assume then
    that the factorization over $\{x,y\}$ doesn't match from the beginning. Now if $u=010$, then $u$ must
    be preceded by $1$ in $t$. Thus we would have found a complete first return to $101$, say $v$, of length $|w|$,
    and the factorization of $v$ over $\{x,y\}$ would match from the beginning. Earlier it was proved that such
    a factor $v$ can't exist.

    Assume then that $|u|$ is even. By \autoref{lem:t_00_11} $u$ contains $00$ or $11$ as a factor, say it contains $00$.
    Suppose that $00$ occurs at an even position in the prefix $v$ of $w$. Then since $|w|$ is odd, it must
    be that $00$ occurs at an odd position in the suffix $v$. Therefore $w$ doesn't match any factorization over
    $\{x,y\}$. If $00$ occurs at an odd position in the prefix $v$, one arrives at a contradiction using
    a symmetric argument.
  \end{proof}

  Now it's also true that the Thue-Morse word doesn't contain any odd palindrome of length greater than three
  (for a proof see \cite{2008:palindromic_lacunas_of_the_thue-morse_word}).
  However the number of even length privileged factors can't be calculated in the same way as we did earlier with Sturmian
  words, since the Thue-Morse word isn't rich. For instance the following factor of $t$ is not rich: $11010011$.

  The case of even length privileged factors is more complicated. So far it's not known to the author how to evaluate
  the number of even length privileged factors in $t$. In the next table there are some values for $\A_n(t)$ for even $n$.
  The results are based on a computer search.

  \vspace{1em}
  \begin{center}
  \begin{tabular}{|c|c|c|c|c|c|c|}
  \hline
  2-10 & 12-20 & 22-30 & 32-40 & 42-50 & 52-60 & 62-70 \\ \hline
  2    & 4     & 4     & 14    & 8     & 0     & 0 \\
  2    & 0     & 8     & 14    & 4     & 0     & 0 \\
  4    & 0     & 8     & 6     & 2     & 0     & 2 \\
  8    & 2     & 4     & 4     & 2     & 0     & 2 \\
  8    & 2     & 6     & 8     & 0     & 0     & 2 \\
  \hline
  \end{tabular}
  \end{center}
  \vspace{1em}

  There are interesting gaps of zeros in $\A_n(t)$. For instance $\A_n(t) = 0$ for
  $81 \leq n \leq 85$, $113 \leq n \leq 117$, $145 \leq n \leq 149$, $177 \leq n \leq 181$ and $189 \leq n \leq 257$.
  Based on the computer searches and an educated guess, we state the following conjecture:

  \begin{conjecture}
    There exist arbitrarily long (but not infinite) gaps of zeroes in the values of $\A_n(t)$.
  \end{conjecture}

  \section{Acknowledgements}
  I thank my advisor Tero Harju for discussions on the practical matters of mathematics and for introducing me
  to combinatorics on words. I also thank my other advisor Luca Zamboni for suggesting the study of privileged words
  and for helping me during my research. Finally I thank Markus Whiteland for useful discussion sessions.

  \printbibliography
  
\end{document}